\documentclass[a4paper,11pt]{article} 

\usepackage{t1enc} 
\usepackage[utf8]{inputenc}
\usepackage{amssymb,amsmath,amsthm} 
\usepackage{graphicx,hyperref}
\usepackage{pgf,tikz}
\usepackage{hyperref}
\usetikzlibrary{shapes,backgrounds,arrows}
\usepackage{mathrsfs}
\usetikzlibrary{arrows}
\newtheorem{theorem}{Theorem}[section]

\newtheorem{corollary}[theorem]{Corollary}
\newtheorem{lemma}[theorem]{Lemma}
\newtheorem{definition}[theorem]{Definition}
\newtheorem{remark}[theorem]{Remark}

\title{On extremal (almost) edge-girth-regular graphs}
\author{}
\date{}

\begin{document}

\author{Gabriela Araujo-Pardo\footnote{Research supported in part by PAPIIT-UNAM-M{\' e}xico IN101821 and Proyecto de Intercambio Académico de la Coordinación de la Investigación Científica de la UNAM "Bipartite biregular cages, block designs and generalized quadrangles".
} \\
garaujo@im.unam.mx
\\
\\
Gy\"orgy Kiss\footnote{This research was supported in part by the Hungarian National Research, Development and Innovation Office  OTKA grant no. K 124950.} \\
gyorgy.kiss@ttk.elte.hu
\\
\\
Istv\'an Porups\'anszki\footnote{This research was supported in part by the Hungarian National Research, Development and Innovation Office  OTKA grant no. SNN 132625.}  \\
rupsansz@gmail.com}
\maketitle
\maketitle
\begin{abstract}
A $k$-regular graph of girth $g$ is called edge-girth-regular graph, shortly egr-graph, if each of its edges is contained in exactly $\lambda$ distinct $g-$cycles. An  egr-graph is called extremal for the triple $(k, g, \lambda)$ if has the smallest possible order. We prove that some graphs arising from incidence graphs of finite planes are extremal egr-graphs. 
We also prove new  lower bounds on the order of egr-graphs.
\end{abstract}

\noindent
\textbf{MSC:} 05C035, 51E20\\
\textbf{Keywords:} edge-girth-regular graph, cage problem, finite biaffine planes 

\section{Introduction}\label{one}

A frequently occurring type of problem in extremal graph theory is the following:
we fix some graph parameters or some graph properties and want to deduce the extremal number of another parameter (in many cases, the number of vertices or edges). 
In this paper, we deal with similar questions, which are also motivated by the 
classical cage problem.
A simple, finite, connected graph $G$ is \emph{$k$-regular} if each vertex has exactly 
$k$ neighbors. It is of  \emph{girth $g$} if its smallest cycles have $g$ vertices. A 
\emph{$g-$cycle} or  \emph{girth cycle} is a cycle of length $g$. The number 
of vertices of $G$ is called the order of $G$. A \emph{$(k,g)-$graph} is a
$k$-regular graph of girth $g$, if it has minimal order, then it is called a 
\emph{$(k,g)-$cage}. For more information about cages we refer to the dynamic 
survey by Exoo and Jajcay \cite{ExooJaj08}.

For $k=2$ the $(k,g)$-graphs are the $g-$cycles, so we assume that $k>2$ throughout this paper. By counting the vertices whose distance from a given vertex (when $g$ is odd) or edge (when $g$ is even) is at most 
$\lfloor (g-1)/2\rfloor $ results in the following bound.
    \begin{theorem}[Moore bound]
    Let $G$ be a $(k,g)$-graph with $k>2$. Then the order of $G$ is at least $n_0(k,g)$, where
    $$n_0(k,g)= 
\begin{cases}
\frac{k(k-1)^{\frac{g-1}{2}}-2}{k-2}, \quad\text{if g is odd;}\\
\frac{2(k-1)^\frac{g}{2}-2}{k-2}, \quad\text{~~if g is even.}
\end{cases}$$
    \end{theorem}

Graphs attaining this bound are called Moore cages.
In \cite{jkm} Jajcay, Kiss, and Miklavi\v c recall a nice
property of Moore cages: the number of $g$-cycles through each of its 
edges is a constant. This observation motivated the following definition. 

\begin{definition}
A $(k,g)$-graph of order $n$ with the property that each of its edges is contained in exactly $\lambda$ distinct $g-$cycles is called \emph{edge-girth-regular graph}, 
shortly an \emph{egr-graph}, and is denoted by $egr(n, k, g, \lambda)$.
\end{definition} 

For a given triple $(k,g, \lambda )$ an $egr(n, k, g, \lambda)$ with minimal order
is the analogue of cages among $(k,g)$-graphs.

\begin{definition}
An $egr(n, k, g, \lambda)$ is called \emph{extremal for the triple $(k, g, \lambda)$} if 
$n$ is the smallest order of any $egr(n, k, g, \lambda)$. We denote the order of this extremal graph with $n(k,g,\lambda)$. If $G$ is an extremal bipartite 
$egr(n,k,g,\lambda),$ then we denote its order with $n_2(k,g,\lambda)$.
\end{definition}

Drglin, Filipovski, Jajcay, and Raiman in \cite{bound} proved lower bounds on the order of edge-girth-regular graphs.

\begin{theorem}[Drglin, Filipovski, Jajcay, Raiman]{\label{lowerbound}}
Let $k\geq 3$ and $g\geq 3$ be a fixed pair of integers, and let 
$\lambda\le(k-1)^\frac{g-1}{2}$, if $g$ is odd and  $\lambda\le(k-1)^\frac{g}{2}$ 
if $g$ is even. Then
$$n(k,g,\lambda)\ge n_0(k,g)+ 
\begin{cases}
(k-1)^\frac{g-1}{2}-\lambda, \quad\text{if g is odd;}\\
\left\lceil2\frac{(k-1)^\frac{g}{2}-\lambda}{k}\right\rceil, \quad\text{~~if g is even.}
\end{cases}$$
Moreover, 
$$n_2(k,g,\lambda)\ge n_0(k,g)+2\left\lceil\frac{(k-1)^\frac{g}{2}-\lambda}{k}\right\rceil.$$
\end{theorem}
For large values of $\lambda ,$ this is the best known general lower bound so far. 
If $\lambda$ is relatively small and $g$ is even, then Porupsánszki \cite{po} improved the bounds using 
eigenvalue techniques.
\begin{theorem}[Porupsánszki]\label{lowerbound_pi}
Let $G$ be an $egr(n,k,g,\lambda)$ graph, where $g$ is even.\\
If $g\equiv0$ (mod $4$), then
$$n(k,g,\lambda)\ge\frac{c(g,k)+k\lambda+k^{g}-2c(\frac{g}{2},k)k^{\frac{g}{2}}}{c(g,k)-c^2(\frac{g}{2},k)+k\lambda},$$
$$n_2(k,g,\lambda)\ge2\frac{c(g,k)+k\lambda+k^{g}-2c(\frac{g}{2},k)k^{\frac{g}{2}}}{c(g,k)-c^2(\frac{g}{2},k)+k\lambda}.$$
If $g\equiv2$ (mod $4$), then
$$n(k,g,\lambda)\ge\frac{c(g,k)+k\lambda+k^g}{c(g,k)+k\lambda},$$
$$n_2(k,g,\lambda)\ge\frac{2k^g}{c(g,k)+k\lambda}.$$
Here $c(\ell,k)$ denotes the number of closed walks of length $\ell$ starting of a vertex 
$v$ that do not contain circles. These formulas are $\frac{\ell}{2}$th degree polynomials of $k$.
\end{theorem}

Let us remark that Theorem \ref{lowerbound} gives 
$O\left(k^{\left[\frac{g}{2}\right]-1}\right)$ lower bounds on the order of extremal $n(k,g,\lambda)$.
On the other hand, if $\lambda=O(k^{\frac{g}{2}-1})$, then the inequalities of 
Theorem \ref{lowerbound_pi} give 
$\Theta\left(k^{\frac{g}{2} }\right)$  lower bounds. Hence if  
$\lambda$ is small, then these lower bounds on the order of extremal egr-graphs 
of even girth larger than the previously known ones. 

The aim of the present paper is on the one hand, to introduce new families of egr-graphs and almost egr-graphs ($(k,g)$-graphs with the property, that each of its edges is
contained in $\lambda _1$ or $\lambda _2$ distinct $g$-cycles, the exact definition is given at the beginning of Section 3) using cleverly chosen substructures of finite geometries, in particular $t$-good structures. 
On the other hand, we prove new lower bounds on the order of (almost) egr-graphs.

In Section 2, we consider graphs arising from incidence graphs of finite planes by the deletion of some specific $t$-good structures. We prove that these graphs are
egr-graphs with small order. In particular, if t=1, then we have extremal egr-graphs, if $t=2$ or $3$, then their order is only a constant higher than the bound given in Theorem \ref{lowerbound}. Finally, we show 
that the point-hyperplane incidence graph of $\mathrm{PG}(n,q),$ the $n$-dimensional finite projective space of order $q$, is an egr-graph of girth 4 for all $n\geq 3, $ moreover it is extremal for $n=3.$

The concept of egr-graphs was generalized by Poto\v cnik and Vidali \cite{pv}, who 
introduced girth-regular graphs. In Section 3, we consider some girth-regular graphs which are almost egr-graphs. We construct graphs of girth $5$ with geometric arguments. It is important to note that this technique of construction has been used previously by a large number of authors. In these papers the authors take the incidence graphs of elliptic semiplanes of type $1$ or $2$ and introduce two operations that in \cite{AABL11} the authors called reduction and amalgam. The first consist in delete set of vertices in the incidence graph that correspond to collinear points and parallel lines in the affine plane, and the second consists in adding a set of edges with some properties in the set of vertices on the incidence graph with the same characteristics. In this paper our arguments are geometric, and we describe a construction for Hoffman-Singleton graph and for a $(6,5)$-graph, after we generalize these constructions. Some of our graphs are graphs originally constructed by Brown \cite{Br-5} and later on by Abreu, Abajo and their co-authors \cite{AABL11,AABB17,ABBM19,AB21}. 
When $q$ is a power of $2,$ we generalize a construction originally due to Araujo-Pardo and Leemans \cite{AL22}, who gave for $q=4$. 
The constructed egr-graphs are related to the cage problem. In some cases, their orders are equal to the order of the smallest known $(k,5)$-graphs.

Finally, in Section 4, we give lower bounds on the order of girth-regular graphs. 
These bounds are generalizations of the lower bounds on the order of extremal egr-graphs proved in \cite{bound} and \cite{po}. In particular, we present a new, combinatorially proved lower bound on the order of girth-regular graphs having even girth.

\section{Edge-girth-regular graphs arising from finite geometries}

In this section, we show that some subgraphs of the point-hyperplane incidence graph of a finite projective space are egr-graphs. These graphs are not new. Many of them 
(and references to their original constructions) can be found in a paper by G\'acs and 
H\'eger \cite{GH08}, however, previously the authors did not examine the edge-girth-regularity of the graphs. Throughout this paper, if $G$ is an incidence graph, then its vertices will be called points and hyperplanes (lines in the planar case) according to whether they correspond to a point or a hyperplane of the geometric structure. For a detailed introduction to  the concepts from finite geometries we use, we refer the reader to the book by Kiss and Sz\H onyi \cite{KSz}. Here we give only the most necessary definitions.

\begin{definition}
Let $\mathcal{S}=(\mathcal{P},\mathcal{L},\mathrm{I})$ be a point-line incidence
geometry, $\mathcal{P}_0\subset \mathcal{P}$ and $\mathcal{L}_0\subset \mathcal{L}$ be  subsets of points and lines in $\mathcal{S}$, respectively. The pair 
$(\mathcal{P}_0,\mathcal{L}_0)$ is called a \emph{$t$-good structure} in $\mathcal{S}$
if there are $t$ lines of $\mathcal{L}_0$ through any point not in $\mathcal{P}_0$,
and  there are $t$ points of $\mathcal{P}_0$ on any line not in $\mathcal{L}_0$,
\end{definition}

In \cite{GH08} the following  general construction method was presented.

\begin{theorem}[G\'acs, H\'eger] 
\label{del-t-good}
Suppose that $(\mathcal{P}_0,\mathcal{L}_0)$ is a $t$-good structure in a generalized 
$n$-gon of order $(q,q).$ Deleting the points and lines of $\mathcal{P}_0$ and 
$\mathcal{L}_0,$ respectively, the incidence graph of the resulting structure is a 
$(q+ 1-t,g)$-graph with $g\geq 2n$.
\end{theorem}

Our first class of egr-graphs related to Baer subplanes. It is well-known that
if a finite projective plane $\Pi _{q^2}$ of order $q^2$ contains a Baer subplane 
$\Pi _{q}$ of ordert $q,$ then the points and lines of $\Pi _{q}$ form a $1$-good structure in $\Pi _{q^2}.$

\begin{theorem}
Suppose that a finite projective plane $\Pi _{q^2}$ of order $q^2$ contains a Baer subplane $\Pi _{q}$ of order 
$q.$ Then there exists an extremal bipartite 
$$egr(2(q^4-q), q^2,6,(q^2-1)(q^4-3q^2+q+2)).$$ 
If $q$ is odd, then it is also an extremal egr-graph.
\end{theorem}

\begin{proof}
Delete all the points and all the lines of $\Pi _q$ from $\Pi _{q^2}$.
We claim that the incidence graph $G$ of the remaining points and lines satisfies the condition of the theorem. The points and the lines of $\Pi _{q}$ form a $1$-good structure, hence $G$ is $q^2$-regular and its girth is at least $6.$. We deleted $q^2+q+1$ points and the same number of lines of $\Pi _{q^2}$, so the order of $G$ is
    $$2(q^4+q^2+1)-2(q^2+q+1)=2(q^4-q).$$

Now, we show that every edge of $G$ is contained in exactly $\lambda $ distinct $6$-cycles where  $\lambda =(q^2-1)(q^4-3q^2+q+2)$. Take an edge $u_1v_1$ and 
count the number of 
$6$-cycles $u_1v_1u_2v_2u_3v_3$. Without loss of generality, we may assume that $u_1$ is a point and $v_1$ is a line of $\Pi _{q^2}$. We have $(q^2-1)$ possible choices for the point $u_2$. After choosing $u_2,$ the choice of $u_3$ uniquely determines the lines $v_2$ and $v_3.$



\begin{figure}[h]
\begin{center}
\includegraphics[scale=0.9]{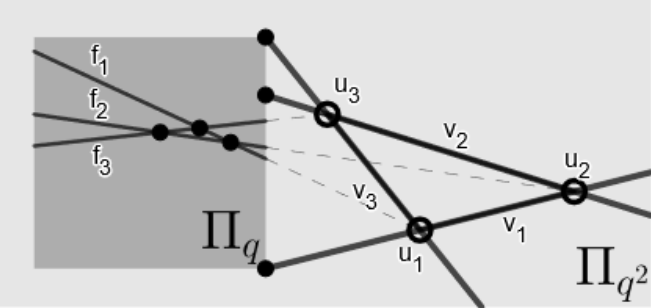}
\caption{Deleted points, Theorem 2.3.}
\label{fig-1}
\end{center}
\end{figure}

Let $f_i$ denote the unique line of $\Pi _{q}$ that is incident with $u_i$. 
A $6-$cycle through the edge $u_1v_1$ contains a deleted element if and only if $u_3$ is incident with $f_1$ or $f_3$. In addition, $u_3$ cannot be incident with the line $v_1$,
because then $u_1v_1u_2v_2u_3v_3$ would not be a $6$-cycle.
The union of these three lines and $\Pi _{q}$ contains 
$$2(q^2-q)+(q^2-2)+(q^2+q+1)=4q^2-q-1$$ 
points  (see Figure \ref{fig-1}), therefore the number of possible choices for the point $u_3$ is
$$(q^4+q^2+1)-(4q^2-q-1)=q^4-3q^2+q+2.$$
So the number of distinct $6-$cycles through any edge is exactly 
$\lambda .$ 

The lower bound in Theorem \ref{lowerbound} gives
$$
\begin{aligned}
    n(q^2,6,(q^2-1)&(q^4-3q^2+q+2))\\
    &\ge n_0(q^2,6)+\left\lceil2\frac{(q^2-1)^3-(q^2-1)(q^4-3q^2+q+2)}{q^2}\right\rceil \\
    &=2(q^4-q^2+1)+2(q^2-q-2)+\left\lceil \frac{2q+2}{q^2}\right\rceil
\geq 2(q^4-q)-1,
\end{aligned}
$$
and
$$n_2(q^2,6,(q^2-1)(q^4-3q^2+q+2))\geq 2(q^4-q).$$
Hence $G$ is an extremal bipartite egr-graph for all $q.$ Moreover,
if $q$ is odd, then the order of a $q$-regular graph is even. So in this case,
$$n(q^2,6,(q^2-1)(q^4-3q^2+q+2))\geq 2(q^4-q),$$
thus $G$ is an extremal egr-graph.
\end{proof}

Let us remark that the points and lines of the union of $t$ disjoint Baer 
subplanes form a $t$-good structure in $\Pi _{q^2}.$ However, for $t>1$ the corresponding incidence graph is not edge-girth-regular.

The next two classes of egr-graphs result from the deletion of degenerate subplanes.

\begin{theorem}
\label{t=2}
Suppose that there exists a finite projective plane $\Pi _q$ of order $q>3$. Then there exists a bipartite 
$$egr(2(q^2-q), q-1,6,(q-2)(q-3)^2).$$
\end{theorem}
\begin{proof}
Take two points $P_1$, $P_2$, the line $e_1=P_1P_2$, and another line $e_2$ that is incident with $P_1$ in $\Pi _q$. Let $\mathcal{L}_0$ be the set of lines that are incident with $P_1$ or $P_2$, and $\mathcal{P}_0$ be the set of points that are incident with $e_1$ or $e_2$. Obviously, $(\mathcal{P}_0,\mathcal{L}_0)$ is a 
$2$-good structure in $\Pi _q$. Hence, by Theorem \ref{del-t-good}, after deleting the
points and lines of $\mathcal{P}_0$ and 
$\mathcal{L}_0$ from $\Pi _q$, the incidence graph $G$ of the resulting structure is a 
$(q-1,g)$-graph with $g\geq 6$. The plane $\Pi _q$ has $q^2+q+1$ points and the same number of lines.
We deleted $2q+1$ lines and $2q+1$ points, so the order of $G$ is 
$$2(q^2+q+1)-2(2q+1)=2(q^2-q).$$

Now, we show that every edge is contained in exactly $\lambda =(q-2)(q-3)^2$ distinct $6-$cycles. Take any edge $u_1v_1$ of $G$ and count the number of $6-$cycles $u_1v_1u_2v_2u_3v_3$. Without loss of generality, we may assume that $u_1$ is a point and $v_1$ is a line of $\Pi _q$. There are $(q-2)$ possible choices for the point $u_2$. 
Again, the choice of $u_3$ uniquely determines the lines $v_2$ and $v_3.$

A $6-$cycle through the edge $u_1v_1$ contains a deleted element if and only if $u_3$ is incident with one of the following lines: $e_1$, $e_2$, $u_1P_1$, $u_1P_2$, $u_2P_1$ or $u_2P_2$. In addition, $u_3$ cannot be incident with the line $u_1u_2$, 
because then $u_1v_1u_2v_2u_3v_3$ would not be a $6$-cycle.
\begin{figure}[h]
\begin{center}
\includegraphics[scale=0.9]{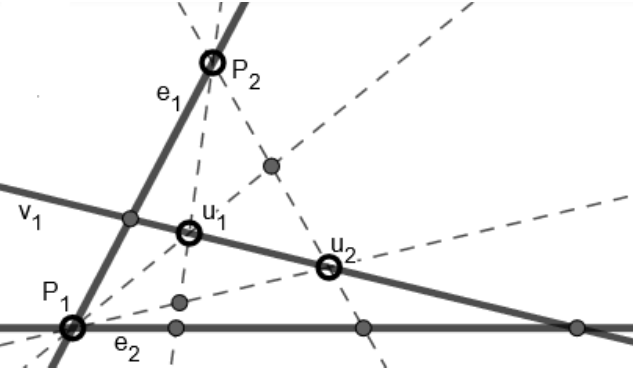}
\caption{Deleted points, Theorem 2.4.}
\label{fig-2}
\end{center}
\end{figure}
The union of these seven lines contains $7q-8$ points (see Figure \ref{fig-2}), therefore the number of possible choices for the point $u_3$ is
$$(q^2+q+1)-(7q-8)=q^2-6q+9=(q-3)^2.$$
So the number of distinct $6-$cycles through any edge is equal to $\lambda .$
As $q>3,$ we have $\lambda >0,$ hence the girth of $G$ is $6.$ Thus $G$ satisfies all conditions of the theorem.
\end{proof}

\begin{theorem}
\label{t=3}
Suppose that there exists a finite projective plane $\Pi _q$ of order $q>4$. Then there exists a bipartite 
$$egr(2(q-1)^2, q-2,6,(q-3)(q^2-9q+21)).$$
\end{theorem}
\begin{proof}
The proof is similar to the proof of the previous theorem.

Take three non-collinear points $P_1$, $P_2$ and $P_3$, and the three lines $e_1$, $e_2$, and $e_3$ that they define in $\Pi _q$. 
Let $\mathcal{L}_0$ be the set of lines that are incident with $P_1,$ $P_2$ or $P_3$, and $\mathcal{P}_0$ be the set of points that are incident with $e_1,$ $e_2$ or $e_3$. Obviously, $(\mathcal{P}_0,\mathcal{L}_0)$ is a 
$3$-good structure in $\Pi _q$. Hence, by Theorem \ref{del-t-good}, after deleting the
points and lines of $\mathcal{P}_0$ and 
$\mathcal{L}_0$ from $\Pi _q$, the incidence graph $G$ of the resulting structure is a 
$(q-2,g)$-graph with $g\geq 6$. We deleted $3q$ lines and $3q$ points of $\Pi _q,$ so the order of $G$ is 
$$2(q^2+q+1-3q)=2(q^2-2q+1)=2(q-1)^2.$$

Now, we show that every edge is contained in exactly $\lambda =(q-3)(q^2-9q+21)$ distinct $6$-cycles. Take any edge $u_1v_1$ and count the number of $u_1v_1u_2v_2u_3v_3$ $6-$cycles. Without loss of generality, we may assume that $u_1$ is a point and $v_1$ is a line of $\Pi _q$. We have $(q-3)$ possible choices for the point $u_2$. Again, the choice of $u_3$ uniquely determines the lines $v_2$ and $v_3.$

\begin{figure}[h]
\begin{center}
\includegraphics[scale=0.7]{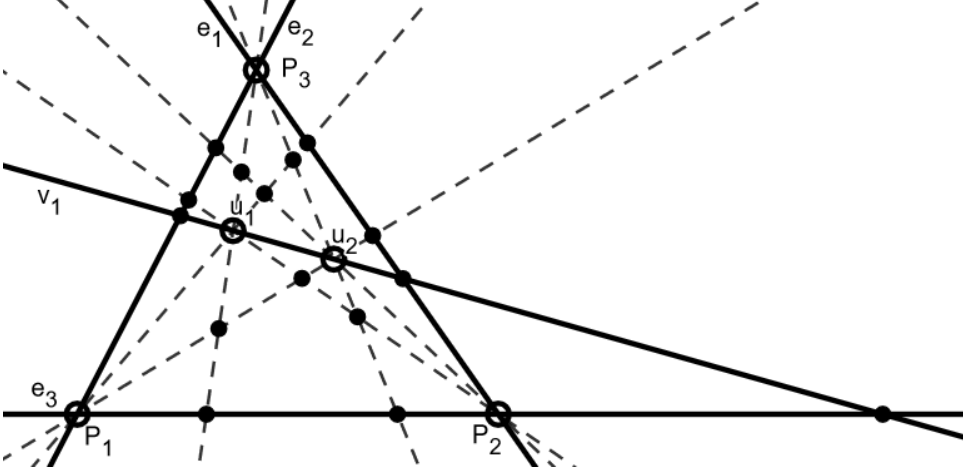}
\caption{Deleted points, Theorem 2.5.}
\label{fig-3}
\end{center}
\end{figure}

A $6-$cycle through the edge $u_1v_1$ contains a deleted element if and only if $u_3$ is incident with one of the following lines: $e_1$, $e_2$, $e_3,$ $u_1P_1$, $u_1P_2$, $u_1P_3$, $u_2P_1$, $u_2P_2$ and $u_2P_3$. In addition, $u_3$ cannot be incident with the line $v_1$ because then $u_1v_1u_2v_2u_3v_3$ would not be a 6-cycle. 
The union of these nine lines contains $10q-20$ points of $\Pi _q$
(see Figure \ref{fig-3}), therefore the number of possible choices for the point $u_3$ is
$$(q^2+q+1)-(10q-20)=q^2-9q+21.$$

So the number of distinct $6-$cycles through any edge is $\lambda .$
The assumption $q>4$ implies $\lambda >0,$ hence the girth of $G$ is $6.$ Thus $G$ satisfies all conditions of the theorem.
\end{proof}

By Theorem \ref{lowerbound}, we have that
$$
\begin{aligned}
    n(q-1,6,(q-2)&(q-3)^2)\\
    &\ge n_0(q-1,6)+\left\lceil 2\frac{(q-2)^3-(q-2)(q-3)^2}{q-1}\right\rceil \\
    &=2(q^2-3q+3)+4q-14+\left\lceil \frac{6}{q-1}\right\rceil  \geq 2q^2-2q-7,
\end{aligned}
$$
and
$$
\begin{aligned}
    n(q-2,6,(q-3)&(q^2-9q+21)\\
    &\ge n_0(q-2,6)+\left\lceil 2\frac{(q-3)^3-(q-3)(q^2-9q+21}{q-1}\right\rceil \\
    &=2(q^2-5q+7)+6q-30+\left\lceil \frac{12}{q-2}\right\rceil  \geq 2q^2-4q-15.
\end{aligned}
$$
So the difference between the lower bound and the order of our example is small.
It is at most $7$ in the case of Theorem \ref{t=2}, and it is at most $17$
in the case of Theorem \ref{t=3}.  

All the known finite projective planes have prime power order and if $q$ is any prime power, then the plane $\mathrm{PG}(2,q)$ has order $q.$ So, according to our current knowledge, the graphs considered in the previous three examples exist if and only if
$q$ is a prime power.

The next example of egr-graphs is based on the properties of the Hermitian curve. 
This is the curve $\mathcal{H}$ with equation 
$X_0^{q+1}+X_1^{q+1}+X_2^{q+1}=0$ in the desarguesian finite projective 
plane $\mathrm{PG}(2,q^2).$ This curve has $q^3+1$ points, at each of its points 
there is a  unique tangent line to $\mathcal{H}$, and all other lines of the plane are
$(q+1)$-secants of $\mathcal{H}$. The points and the tangent lines of $\mathcal{H}$ form a $(q+1)$-good structure in $\mathrm{PG}(2,q^2).$

\begin{theorem}
Let $q$ be a prime power. Then there exists a  bipartite 
$$egr\left(2(q^4-q^3+q^2),q^2-q,6,(q^2-q-1)(q^4 - 3 q^3 - q^2 + 5q + 3)\right).$$
\end{theorem}

\begin{proof}
Let $\mathcal{H}$ be a Hermitian curve in $\mathrm{PG}(2,q^2).$
Delete the points and the tangent lines of $\mathcal{H}$ from $PG(2,q^2)$. 
Let $G$ denote the incident graph of the remaining points and lines. 
The points and the tangent lines of $\mathcal{H}$ form a $(q+1)$-good
structure, hence, by Theorem \ref{del-t-good}, $G$ is $(q^2-q)$-regular and its girth 
is at least $6.$ The order of $G$ is 
$$2(q^4+q^2+1)-2(q^3+1)=2(q^4-q^3+q^2).$$

Now, we count the number of $6$-cycles $u_1v_1u_2v_2u_3v_3$ through an edge $u_1v_1$. Without loss of generality, we may assume that $u_1$ is a point and $v_1$ is a line. We have $q^2-q-1$ possible choices for the line $v_3$, and $(q^2-q-1)^2$ possible choices for the pair $(u_2,u_3)$. However, the line $v_2$ may be a tangent line of $\mathcal{H}$. The tangent lines through the point $u_1$ intersect $\mathcal{H}$
in exactly $(q+1)$ collinear points and the lines $v_1$ and $v_3$ also intersect 
$\mathcal{H}$ in $(q+1)$ points. All points of $\mathcal{H}$ have the property that  the unique tangent lines at those points intersect the line $v_1$ and $v_3$ not in 
$\mathcal{H}$ and these  points of intersection are different from the point $u_1$. 
Therefore the exact number of possible choices for the pair $(u_2,u_3)$ is
$$(q^2-q-1)^2-(q^3+1-3(q+1))=q^4 - 3 q^3 - q^2 + 5q + 3.$$
Hence the number of $6$-cycles through the edge $u_1v_1$ is
$$(q^2-q-1)(q^4 - 3 q^3 - q^2 + 5q + 3).$$

So $G$ satisfies all conditions of the theorem.
\end{proof}

\begin{remark}
{\rm Theorem \ref{lowerbound} gives
$$n(q^2-q,6,(q^2-q-1)(q^4 - 3 q^3 - q^2 + 5q + 3))\ge 2(q^4-q^3-3q-1).$$
so the order of $G$ is close to the bound.}
\end{remark}

\begin{theorem}
Let $G$ be the point-hyperplane incidence graph of $\mathrm{PG}(n,q),$ the $n$-dimensional finite projective space of order $q$, where $n\ge3$. Then $G$ is a bipartite 
$$egr\left(2\frac{q^{n+1}-1}{q-1},\frac{q^n-1}{q-1},4,
\frac{q^{2n-1}-q^{n+1}-q^{n}+q^2}{(q-1)^2}\right).$$

If $n=3,$ then $G$ is an extremal bipartite egr-graph.
\end{theorem}
\begin{proof}
The only part we do not obtain immediately from the basic combinatorial properties of 
$\mathrm{PG}(n,q),$ is the exact value of $\lambda ,$ the number of $4$-cycles through an edge. Consider an incident point-hyperplane pair $(P_1,\mathcal{S}_1)$ and count the number of $4-$cycles $P_1S_1P_2S_2$ through this edge. For the point $P_2$ we have 
$\frac{q^{n}-1}{q-1}-1$ possible choices. There are $\frac{q^{n-1}-1}{q-1}$ hyperplanes that contain the points $P_1$ and $P_2$, one of them is $\mathcal{S}_1$, hence the number of possible choices for the hyperplane $\mathcal{S}_2$ is 
$\frac{q^{n-1}-1}{q-1}-1.$ 
Therefore the number of girth cycles through any edge is 
$$\left( \frac{q^{n}-1}{q-1}-1\right) \left( \frac{q^{n-1}-1}{q-1}-1\right) =
\frac{q^{2n-1}-q^{n+1}-q^{n}+q^2}{(q-1)^2}.$$

If $n=3$, then Theorem \ref{lowerbound_pi} gives the lower bound 
    $$n_2(q+1,4,q^3+q^2)\ge2(q^3+q^2+q+1).$$
So $G$ is an extremal bipartite egr-graph, because its order attains the bound.
\end{proof}

\begin{remark}\label{compare_4}
{\rm The graph $G$ is relatively simple, the interesting part is its extremality
for $n=3,$  because it shows the strength of Theorem \ref{lowerbound_pi} for small 
values of $\lambda .$ For comparison, Theorem \ref{lowerbound} gives a significantly weaker lower bound:
    $$n_2(q+1,4,q^3+q^2)\ge 2(2q^2+q+1).$$}
\end{remark}

\section{Almost edge-girth-regular graphs}
The definition of egr-graphs
was weakened by Poto\v cnik and Vidali \cite{pv} in the following way. 
One can introduce the signature
$(a_1,a_2,\ldots a_k)$ of a vertex as the ordered sequence of the number of girth 
cycles containing the edges emanating from the vertex. A
graph is called \emph{girth-regular} if all of its points have the same
signature. Girth-regular graphs are regular graphs. If the signature of a 
girth-regular graph $G$ satisfies $a_1=a_2=\ldots =a_k$, then $G$ is an egr-graph.
 If the signature of $G$ has exactly two different entries, then we call $G$ an
\emph{almost egr-graph}, shortly an \emph{agr-graph}. When an agr-graph $G$ of order $v$ has girth $g$, the two entries of its signature are $a$ and $b$ which appear $k_1$ and
$k_2$ times, respectively, then we say that $G$ is an
$$\mathrm{agr}\left( v,k_1+k_2,g,\left[a_{(k_1)},b_{(k_2)}\right]\right) .$$

We can construct signature-regular graphs of girth $5$ from the incidence graphs of 
biaffine planes. 
\begin{definition}
Let $\Pi$ be a finite projective plane of order $q$. A biaffine plane is obtained from $\Pi$ by choosing a point-line pair $(P,\ell )$ and deleting $P$, $\ell$, all the lines incident with $P$ and all the points belonging to $\ell$. If the point-line pair is incident in $\Pi$, then we call the biaffine plane type 1, otherwise, type 2.
\end{definition}

Starting from the incidence graph of a biaffine plane, we just need to add some new edges such that the new graph has girth $5,$ and all new edges are contained in the same number of distinct girth cycles. 
In this section we always consider $\mathrm{PG}(2,q)$ as 
$\mathrm{AG}(2,q) \cup \ell_{\infty}.$  We delete the line at infinity and a pencil of lines of $\mathrm{AG}(2,q)$. The obtained biaffine plane is denoted by $\mathcal{B}$, and we coordinatize $\mathcal{B}$ by Cartesian coordinates in the usual way. The incidence graph of $\mathcal{B}$ is denoted by $G.$

First, we present a general construction method.

\begin{lemma}
\label{gen-constr}
Let $p>3$ be a prime and $q=p^r.$ Then there exists a 
$(q+2,5)$-graph of order $2q^2.$
\end{lemma}
\begin{proof}
Delete the vertical lines of $\mathrm{AG}(2,q)$. Then $\mathcal{B}$ is a
biaffine plane of type $1$.  Let 
$\varepsilon\in GF(q)\setminus \{ 0,\pm1\} .$ We define new 
edges among the points and lines of $\mathcal{B}$ in the following way. 
The two neighbors of the point $(x,y)$ are the points $(x,y\pm1)$ on the same vertical line, and the two neighbors of the line 
$Y=mX+b$ are the lines $Y=mX+b\pm\varepsilon $ from the same parallel class.
As the characteristic of $\mathrm{GF}(q)$ is $p,$ the new edges form $p$-cycles. 
Let $\Gamma $ denote the graph $G$ extended by the new edges.
By definition, $\Gamma $ has $2q^2$ vertices,  $q^2$ points and $q^2$ lines. Each vertex has $2$ neighbors of its type and $q$ neighbors of the other type, so $\Gamma $
is $(q+2)$-regular.

We claim that $\Gamma $ has girth $5.$ 
Suppose that 
$\Gamma $ contains a cycle $\mathcal{C}$ of length $\ell \leq 4$. Then 
$\mathcal{C}$ must contain a new edge. A new edge joins either two points on a vertical line of $\mathrm{AG}(2,q)$, or two parallel lines of $\mathrm{AG}(2,q)$. These pairs of vertices have no common neighbors in $G,$ and $p>3,$ so $\Gamma $ is triangle-free. If 
$\mathcal{C}$ is a four-cycle $u_1u_2v_2v_1$ where $u_1u_2$ is a new edge, 
then there are three possibilities. First, if all four vertices of $\mathcal{C}$ were the same type, then $p=4,$ a contradiction. If three 
vertices were the same type, then without loss of generality we may assume, that
$v_2$ is the only vertex of the other type. Then $u_1$ and $v_1$ were the two
same-type neighbors of $u_2,$ so they would be collinear points or parallel 
lines of $\mathrm{AG}(2,q)$, hence they could not have a common neighbor in the 
incidence graph of $\mathcal{B}$. Hence $v_2v_1$ is also a new edge and the types of $u_1$ and $v_1$ cannot be the same. We may assume without loss of generality that $u_1=(c,d)$ and $u_2=(c,d+1)$ are points, $v_1$ and $v_2$ are lines having equations $Y=mX+b$ and
$Y=mX+b\pm\varepsilon $, respectively. As the point $u_i$ is on the line $v_i,$ we get
$d=mc+b$ for $i=1,$ and $d+1=mc+b\pm\varepsilon $ for $i=2$. This is a contradiction because $\varepsilon \neq \pm1.$ Thus the girth of $\Gamma $ is at least $5.$ The order of $\Gamma $ is less than $n_0(q+2,6),$ so its girth is 5.
\end{proof}

\begin{corollary}
\label{HS+40}
The Hoffman-Singleton graph and the $(6,5)$-cage of order $40$ have a simple, geometric
construction. 
\end{corollary}
\begin{proof}
Let $q=5$ and $\varepsilon =2.$ Then the graph $\Gamma $ constructed in Lemma\ref{gen-constr} is a $(7,5)$-graph of order $50,$ so it is the Hoffman-Singleton graph,
which is a Moore cage.

The five points on the vertical line $X=0$ and the five horizontal lines of 
$\mathrm{AG}(2,5)$ form a $1$-good structure in $\mathcal{B}$. Although  
$\mathcal{B}$ is not a generalized polygon, it is clear that Theorem \ref{del-t-good}
holds in this case as well. Deleting these points and lines from $\Gamma ,$ the resulting graph is a $(6,5)$-graph of order $40.$  The uniqueness of this graph was proven by Wong \cite{wo}, so it is the $(6,5)$-cage. 
\end{proof}

\begin{remark}
{\rm  A computer-assisted calculation shows that the $(6,5)$-cage is an 
egr$(40,6,5,22).$ Its order attains the bound of Theorem \ref{lowerbound}, so it is an
extremal egr-graph. We can also prove this fact using purely combinatorial arguments. 
We omit the proof, because it is straightforward, but long counting.}
\end{remark}

\begin{theorem}
\label{korok1}
Let $p>5$ be a prime and $q=p^r \ge 11$ be a prime power. Then there exists an 
$$\mathrm{agr}\left(2q^2,q+2,5,\left[8(q-1)_{(q)},(q^2-q)_{(2)}\right]\right).$$
\end{theorem}

\begin{proof}
As $q \ge 11,$ we can choose $\varepsilon $ so that
$\varepsilon \notin \{ 0,1, \pm 2, \pm \frac{1}{2} \} .$ Consider the 
$(q+2,5)$-graph $\Gamma $ constructed in Lemma \ref{gen-constr}. We claim that
$\Gamma $ is an agr-graph.

First, we show that a $5$-cycle in $\Gamma $ cannot contain three consecutive vertices of the same type. The assumption $p>5$ implies that
a $5$-cycle cannot contain five vertices of the same type. If exactly four vertices were same type in a $5$-cycle $u_1u_2u_3u_4v_1,$ then $u_1$ and $u_4$ were two points on a vertical line of $\mathrm{AG}(2,q)$, or two parallel lines of 
$\mathrm{AG}(2,q).$ Thus $v_1$ would be a vertical line, or a point on the line at infinity, respectively. But these elements were deleted from the incidence graph of
$\mathrm{PG}(2,q).$ If exactly three consecutive vertices were same type in a $5$-cycle $u_1u_2u_3v_1v_2,$ then we may assume without loss of generality that $u_1=(c,d)$ and $u_3=(c,d+2)$ are points, $v_1$ and $v_2$ are lines having equations $Y=mX+b$ and $Y=mX+b\pm\varepsilon $, respectively. The same calculations as in the proof of
Lemma \ref{gen-constr} show, that these imply $\varepsilon =\pm 2,$ while the assumption that $u_1$ and $u_3$ are lines and $v_1$ and $v_2$ are points implies 
$\varepsilon =\pm \frac{1}{2}.$ 

Now, we count the number of $5$-cycles through an edge. 
First, consider an edge $u_1v_1$ where $u_1$ is a point and $v_1$ is a line. 
There are four types of $5-$cycles through the edge $u_1v_1$: $u_1v_1u_2u_3v_2$, $u_1v_1u_2v_2v_3$, $u_1v_1u_2v_2u_3$, and $u_1v_1v_2u_2v_3$, where $u_i$ denotes a point and $v_j$ denotes a line of $\mathcal{B}$. 
Consider the first type of $5-$cycles. We have $q-1$ possible choices for the point $u_2$, and two possible choices for the point $u_3$. The line $v_2$ is determined, it is the line joining the points $u_1$ and $u_3$.
In the second case, we have $q-1$ possible choices for the line $v_3$, and two possible choices for the line $v_2$ The point $u_2$ is the point of intersection of the lines $v_1$ and $v_2$. In the third type of $5$-cycles, we have two possible choices for the point $u_3$. In this case, we have $q-1$ possible choices for the point $u_2$ and the line $v_2$ is the line joining the points $u_2$ and $u_3$.
Finally, consider a $5$-cycle of the fourth type. We have $q-1$ possible choices for the line $v_3$, and two possible choices for the line $v_2$. The point $u_2$ is the point of intersection of the lines $v_2$ and $v_3$.
In summary, there are $8(q-1)$ distinct girth cycles through the edge $u_1v_1$.

Now, consider the girth cycles through a new edge joining the points $u_1$ and $u_2$. This time, we have only one type of $5$-cycles: $u_1u_2v_1u_3v_3$. 
The cycle is uniquely determined by $u_3,$ which is an arbitrary point not on the 
vertical line through $u_1$ and $u_2.$
Hence there are $q^2-q$ distinct girth cycles through the edge $u_1u_2$. Similarly, the new edge joining the lines $v_1$ and $v_2$ is contained in exactly $q^2-q$ distinct 
$5$-cycles.

As every vertex of $\Gamma$ has exactly $q$ neighbors of the same type and two 
neighbors of the other type, the signature of $\Gamma$ is 
$$
  \left[8(q-1)_{(q)},(q^2-q)_{(2)}\right].  
$$
\end{proof}

\begin{remark}
{\rm 
For $q=7$ and $\varepsilon =2$, Lemma \ref{gen-constr} results in a $(9,5)$-graph of order $98$. It has 2 more vertices than the smallest known $(9,5)$-graph, which was 
constructed by J\o rgensen \cite{jorg}.
However, in this case, there are 5-cycles containing three 
consecutive points and $5-$cycles containing three consecutive lines. Any of these 
possibilities breaks the symmetry of points and lines, so the constructed graph is not an agr-graph.}
\end{remark}

\begin{theorem}
Let $q=5^r$ be a power of $5$. Then there exists an 
$$\mathrm{agr}\left(2q^2,q+2,5,\left[8q-4_{(q)},(q+1)^2_{(2)}\right]\right).$$
\end{theorem}

\begin{proof}
Let $\varepsilon =2$ and consider the $(q+2,5)$-graph $\Gamma $ constructed in Lemma \ref{gen-constr}. We claim that
$\Gamma $ is an agr-graph.

First, consider an edge $u_1v_1$, such that $u_1$ is a point and $v_1$ is a line. There are four types of $5-$cycles through $u_1v_1$ that have at most two consecutive vertices of the same type: $u_1v_1u_2u_3v_2$, $u_1v_1u_2v_2v_3$, $u_1v_1u_2v_2u_3$, and $u_1v_1v_2u_2v_3$, where $u_i$ denotes a point and $v_j$ denotes a line of $\mathcal{B}$. From the proof of the previous theorem, we know that the number of these girth cycles is $8(q-1)$. Unlike the previous theorem, we have girth cycles that contain exactly three consecutive points or lines. Assume that $u_1=(c,d)$ and the line $v_1$ has equation $Y=mX+b$. There are four distinct girth cycles through the edge $u_1v_1$ 
that contain three consecutive points or lines:

\begin{enumerate}
    \item $v_2\colon Y=mX+b+2,\quad v_3\colon Y=mX+b+4,\quad u_2=(c,d-1),$
    \item $v_2\colon Y= mx+b-2,\quad v_3\colon Y=mX+b-4,\quad u_2=(c,d+1),$
    \item $v_2\colon Y=mX+b-2,\quad u_2=(c,d-2),\quad u_3=(c,d-1),$
    \item $v_2\colon Y=mX+b+2,\quad u_2=(c,d+2),\quad u_3=(c,d+1). $
\end{enumerate}
There are no girth cycles with exactly four points or lines because $3\cdot (\pm 2)\neq 0$, hence the number of girth cycles through the edge $u_1v_1$ is exactly 
$$8(q-1)+4=8q-4.$$

The other type of edges comes from the cycles. Consider the girth cycles through the edge $u_1u_2$. Similarly to the previous theorem, the edge $u_1u_2$ is contained in exactly $q^2-q$ distinct $5$-cycles that do not contain three or more consecutive points or lines.
We have three types of $5-$cycles through the edge $u_1u_2$ that have exactly three consecutive vertices of the same type: $u_1u_2u_3v_1v_2$, $u_1u_2v_1v_2u_3$, and $u_1u_2v_1v_2v_3$. Consider the first type of these $5-$cycles. There are $q$ possible choices for the line $v_2$, and the point $u_2$ and line $v_1$ are uniquely determined by the other components of the cycle. In the second case, we have $q$ possible choices for the line $v_1$, while the rest of the cycle is uniquely determined again.
In the third type of $5$-cycles we have $q$ possible choices for the line $v_1$. In this case, the lines $v_2$ and $v_3$ are uniquely determined.
Since the characteristic of the field is 5, we also have $5-$cycles that contain only one type of vertices. Every edge joining two vertices of the same type is contained in exactly one such cycle.
In summary, the number of girth cycles containing the edge $u_1u_2$ is
$$q(q-1)+3q+1=(q+1)^2.$$

Every vertex of $\Gamma$ has exactly $q$ neighbors of the same type and two of the other type, hence the signature of $\Gamma$ is 
$$
  \left[8q-4_{(q)},(q+1)^2_{(2)}\right].  
$$
In particular, if $q=5,$ then $8q-4=(q+1)^2,$ so  $\Gamma$ is an egr-graph, as we have already seen in Corollary \ref{HS+40}.
\end{proof}

\begin{theorem}\label{korok2}
Let $q\ge11$ be a prime power. Then there exists an 
$$agr\left(2(q^2-1),q+2,5,\left[8(q-1)_{(q)},(q^2-q)_{(2)}\right]\right).$$
\end{theorem}
\begin{proof}
Delete the lines of $\mathrm{AG}(2,q)$ through the origin. Then $\mathcal{B}$ is a biaffine plane of type $2$ 
Choose two generators,  $\varepsilon $ and $\eta $ of the multiplicative group
$GF(q)^*$ so that $\varepsilon \neq  \eta ^{\pm 1}, \varepsilon\neq \eta ^{\pm 2}, \eta\neq \varepsilon^{\pm 2}.$ Such generators exist because $q>8$. 

We define a cycle on the $q-1$ points on each deleted line of $\mathrm{AG}(2,q)$, and a cycle on the $q-1$ lines in each parallel class of lines of $\mathcal{B}$ in the following way:
If $P=(x,y)$ is a point on a line $\ell $ through the origin, then the set of the $q-1$ points of $\mathcal{B}$ on $\ell $ is
$$\{ P^i=(x\varepsilon ^i,y\varepsilon ^i) \colon i=1,2,\dots ,q-1\} .$$
Join the points $P^i$ and $P^{i+1}$ for all $i$ where the superscripts are taken modulo
$q-1.$ As $\varepsilon $ is a generator, these new edges form a cycle of length $q-1.$ 
If $e \colon AX+BY+1=0$ is a line of $\mathcal{B}$ in a parallel class $P_e$,
then the set of the $q-1$ lines of $\mathcal{B}$ in $P_e$ is
$$\{ e^i: AX+BY+\eta ^i=0 \colon i=1,2,\dots ,q-1\} .$$
Join the lines $e^i$ and $e^{i+1}$ for all $i$ where the superscripts are taken modulo
$q-1.$ As $\eta $ is a generator, these new edges again form a cycle of length $q-1.$
Let $\Gamma $ denote the graph $G$ extended by the edges of these $2(q+1)$ cycles. 
By definition, $\Gamma $ has $2(q^2-1)$ vertices,  $q^2-1$ points and $q^2-1$ lines. Each vertex has $2$ neighbors of its type and $q$ neighbors of the other type, so
$\Gamma $ is $(q+2)$-regular.

We claim that $\Gamma $ has girth $5.$ 
Suppose that $\Gamma $ contains a cycle $\mathcal{C}$ of length $\ell \leq 4$. 
Then $\mathcal{C}$ must contain a new edge. If all four vertices of $\mathcal{C}$ 
were the same type, then the order of $\varepsilon $ or $\eta $ would be $4<q-1,$ a contradiction. If three vertices were the same type, then the same reasoning as in the proof of Lemma \ref{gen-constr} works.
Finally, if $u_i$ are points and $v_i$ are lines, then without loss of generality we may assume that $u_1=(c,d)$ and 
$u_2=(c\varepsilon ,d\varepsilon )$ are points, $v_1$ and $v_2$ are lines having equations $AX+BY+F=0$ and $AX+BY+F\eta ^{\pm 1}=0 $, respectively. As the point $u_i$ is on the line $v_i,$ we get
$$Ac+Bd+\eta ^{i}=0 \text{  and  }Ac\varepsilon +Bd\varepsilon +\eta ^{i\pm 1}=0.$$
Hence $\varepsilon \eta ^{i} -\eta ^{i\pm 1} =0,$ so $\varepsilon =\eta ^{\pm 1},$
a contradiction again. 
Thus the girth of $\Gamma $ is at least $5.$ The order of
$\Gamma $ is less than $n_0(q+1,6),$ so its girth is 5.

Now, we count the number of girth cycles through each edge. First, consider an edge $u_1v_1$, such that $u_1$ is a point and $v_1$ is a line. In any cycle there are at most two consecutive vertices of the same type (due to the choice of $\varepsilon$ and $\eta$), hence we have four types of $5-$cycles through the edge $u_1v_1$: $u_1v_1u_2u_3v_2$, $u_1v_1u_2v_2v_3$, $u_1v_1u_2v_2u_3$, and $u_1v_1v_2u_2v_3$, where $u_i$ denotes a point and $v_j$ denotes a line of $\mathcal{B}$.

The counting is very similar to the one presented in Theorem \ref{korok1}. 
Consider the first type of $5$-cycles. We have $q-1$ possible choices for the point $u_2$, and two possible choices for the point $u_3$. Then $v_2$ is the line joining the points $u_1$ and $u_3$.
In the second case, we have $q-1$ possible choices for the line $v_3$, and two possible choices for the line $v_2$ Then $u_2$ is the point of intersection of $v_1$ and $v_2$.
In the third type of $5$-cycles we have two possible choices for the point $u_3$. In this case, we have $q-1$ possible choices for the point $u_2$ and $v_2$ is the line
joining the points $u_2$ and $u_3$.
Finally, consider a $5$-cycle of the fourth type. We have $q-1$ possible choices for the line $v_3$, and two possible choices for the line $v_2$. The point $u_2$ is the point of intersection of the lines $v_2$ and $v_3$.
In summary, there are $8(q-1)$ distinct girth cycles through the edge $u_1v_1$.

The counting of the girth cycles through a new is exactly the same as in the proof of Theorem \ref{korok1}. So each new edge is contained in $q^2-q$ distinct $5$-cycles.
Every vertex of $\Gamma$ has exactly $q$ neighbors of the same type and two of the other type. Hence the signature of $\Gamma$ is 
$$
  \left[8(q-1)_{(q)},(q^2-q)_{(2)}\right].  
$$
\end{proof}

\begin{remark}
{\rm Again, this construction works and results in a $(2q^2-2,5)$-graph if there are two distinct generators of the multiplicative group $GF(q)^*$ such that their product is not the unit element, so for all $q\ge 8.$ 
The problem is the same as in Theorem \ref{korok1}: there are $5-$cycles with three vertices of the same type next to each other. This breaks the symmetry of points and lines and we obtain a graph that is not an agr-graph. For $q=8$ the order of the constructed graph is 126. It has 2 more vertices than the smallest known $(9,5)$-graph, which
was discovered by Exoo  \cite{ex-126}. He found it by a computer search starting from  the incidence graph of $\mathrm{PG}(2,11).$
}  
\end{remark}

In the next two constructions, the new edges are perfect matchings of the points and 
of the lines.
\begin{theorem}
\label{parositastad}
Let $q>2$ be an even prime power. Then there exists an 
$$\mathrm{agr}\left(2q^2,q+1,5,\left[4(q-1)_{(q)},(q^2-q)_{(1)}\right]\right).$$
\end{theorem}
\begin{proof}
Delete the vertical lines of $\mathrm{AG}(2,q)$. Then $\mathcal{B}$ is a biaffine plane of type $1$. Let 
$\varepsilon\in GF(q)\setminus \{ 0,1\} $. Now, we define a perfect matching of the points of $\mathcal{B}:$ 
    $$(x,y)\longleftrightarrow(x,y+1),$$
    and a perfect matching of the lines of $\mathcal{B}:$
    $$Y=mX+b\longleftrightarrow Y=mX+b+\varepsilon.$$
    These are bijections because $q$ is even. 
Let $\Gamma $ denote the graph $G$ extended by the new edges of these two 
matchings.
We claim that $\Gamma $ satisfies the conditions of the theorem. 

Since $\Gamma$ is a $(q+1)$-regular subgraph of the graph presented in Theorem \ref{korok1}, we only need to count the number of $5$-cycles through each edge.
First, consider an edge $u_1v_1$, such that $u_1$ is a point and $v_1$ is a line. In any cycle there are at most two consecutive vertices of the same type, hence we have four types of $5$-cycles through the edge $u_1v_1$: $u_1v_1u_2u_3v_2$, $u_1v_1u_2v_2v_3$, $u_1v_1u_2v_2u_3$, and $u_1v_1v_2u_2v_3$, where $u_i$ denotes a point and $v_j$ denotes a line of $\mathcal{B}$. There are $q-1$ distinct 
$5$-cycles for each type.
The other type of edges comes from the matchings. Again, any edge of this type is contained in exactly $q^2-q$ distinct $5$-cycles.

Every vertex of $\Gamma$ has exactly $q$ neighbors of the same type and one neighbor of the other type. Hence the signature of $\Gamma$ is 
$$
  \left[4(q-1)_{(q)},(q^2-q)_{(1)}\right].  
$$
\end{proof}

\begin{remark}
{\rm If $q=4$, then $\Gamma $ is the extremal $egr(32,5,5,12)$ graph constructed by
Araujo-Pardo and Leemans \cite{AL22}.}
\end{remark}

\begin{theorem}
Let $q>5$ be an odd prime power. Then there exists an 
$$\mathrm{agr}\left(2(q^2-1),q+1,5,\left[4(q-1)_{(q)},(q^2-q)_{(1)}\right]\right).$$
\end{theorem}
\begin{proof}
Delete the lines through the origin of $\mathrm{AG}(2,q)$. Then $\mathcal{B}$ is a biaffine plane of type $1$. Choose two generators  $\varepsilon $ and $\eta $ of the multiplicative group $GF(q)^*$ such that $\varepsilon \neq  \eta ^{\pm 1} .$
We define a perfect matching on the $q-1$ points on each deleted line of 
$\mathrm{AG}(2,q)$, and a perfect matching on the $q-1$ lines in each parallel class of lines of $\mathcal{B}$ in the following way.
If $P=(x,y)$ is a point on a line $\ell $ through the origin, then the set of the $q-1$ points of $\mathcal{B}$ on $\ell $ is
$$\{ P^i=(x\varepsilon ^i,y\varepsilon ^i) \colon i=1,2,\dots ,q-1\} .$$
Join the points $P^i$ and $P^{i+1}$ for all \underline{even} $i$ where the superscripts are taken modulo
$q-1.$ As $\varepsilon $ is a generator, we get a perfect matching.

If $e \colon AX+BY+1=0$ is a line of $\mathcal{B}$ in a parallel class $P_e$,
then the set of the $q-1$ lines of $\mathcal{B}$ in $P_e$ is
$$\{ e^i: AX+BY+\eta ^i=0 \colon i=1,2,\dots ,q-1\} .$$
Join the lines $e^i$ and $e^{i+1}$ for all \underline{even} $i$ where the superscripts are taken modulo
$q-1.$ As $\eta $ is a generator, we get a perfect matching.

Let $\Gamma $ denote the graph $G$ extended by the new edges of these $2(q+1)$ cycles. We claim that $\Gamma $ satisfies the conditions of the theorem. We can apply the same argument as in the proof of Theorem \ref{parositastad}. Hence the signature of $\Gamma$ is 
$$
  \left[4(q-1)_{(q)},(q^2-q)_{(1)}\right].  
$$ 
\end{proof}

\section{Lower bounds on the order of girth-regular graphs}

In this section, we present lower bounds on the order of girth-regular graphs. We present a natural extension of the known bounds on the order of extremal egr-graphs to 
girth-regular graphs and give a purely combinatorial proof on the lower bound for even girth. Let $sgr(n,k,g,\mathbf{a})$ denote a girth-regular graph of order $n$, valency $k$, girth $g$ and signature $\mathbf{a}=(a_1,a_2,\ldots, a_k). $

\begin{theorem}
    Let $G$ be an $srg(n,k,g,\mathbf{a})$ graph, where $g=2h$ is an even number. Then
    $$n\ge 2\frac{(k-1)^h-2}{k-2}+\left\lceil\frac{(k-1)^h-2a_1}{k}\right\rceil.$$
\end{theorem}
\begin{proof}
Same argument as in Theorem \ref{lowerbound} but with an edge that is contained in exactly $a_1$ distinct girth cycles.
\end{proof}
\begin{theorem}
    Let $G$ be an $srg(n,k,g,\mathbf{a})$ graph, where $g=2h+1$ is an odd number. Then
    $$n\ge \frac{k(k-1)^h-2}{k-2}+\left\lceil\frac{k(k-1)^h-\sum_{i=1}^ka_i}{k}\right\rceil.$$
\end{theorem}
\begin{proof}
    Same argument as in Theorem \ref{lowerbound} with changing $k\lambda$ to $\sum_{i=1}^ka_i.$
\end{proof}
\begin{theorem}
Let $G$ be an $sgr(n,k,g,\mathbf{a})$ graph, where $g$ is even.\\
If $g\equiv0$ (mod $4$), then
$$n(k,g,\mathbf{a})\ge\frac{c(g,k)+ \sum_{i=1}^ka_i+k^{g}-2c(\frac{g}{2},k)k^{\frac{g}{2}}}{c(g,k)-c^2(\frac{g}{2},k)+ \sum_{i=1}^ka_i},$$
$$n_2(k,g,\mathbf{a})\ge2\frac{c(g,k)+ \sum_{i=1}^ka_i+k^{g}-2c(\frac{g}{2},k)k^{\frac{g}{2}}}{c(g,k)-c^2(\frac{g}{2},k)+ \sum_{i=1}^ka_i}.$$
If $g\equiv2$ (mod $4$), then
$$n(k,g,\mathbf{a})\ge\frac{c(g,k)+ \sum_{i=1}^ka_i+k^g}{c(g,k)+ \sum_{i=1}^ka_i},$$
$$n_2(k,g,\mathbf{a})\ge\frac{2k^g}{c(g,k)+ \sum_{i=1}^ka_i}.$$
\end{theorem}
\begin{proof}
Again, the proof is the same as the proof of Theorem \ref{lowerbound_pi}. We only need to change $k\lambda$ to $\sum_{i=1}^ka_i$.
\end{proof}
\begin{theorem}{\label{cycles}}
Let $G$ be an $srg(n,k,g,\mathbf{a})$ graph, where $g=2h$ is an even number. Then
$$
    \begin{aligned}
        n\ge & 2\frac{(k-1)^h-1}{k-2} \\
        & +\max_{1\le i\le k}\left\lceil\frac{\left((k-1)^h-a_i\right)^2}{\sum_{j=1}^ka_j-3a_i+(k-1)^h-\max\left(0,\left\lceil\frac{a_i}{2(k-1)^{(h-1)}}-\frac{a_i}{2}\right\rceil\right)}\right\rceil.
    \end{aligned}
$$
\end{theorem}
\begin{proof}
Choose an arbitrary edge $uv$ that is contained in exactly $\lambda$ distinct $g-$cycles and define the sets $D_i(u)$ of vertices as follows: $w\in D_i(u)$ if the length of the shortest $uw-$path is $i$ and the length of the shortest $vw-$path is $i+1$. Similarly,  $w\in D_i(v)$ if and only if the length of the shortest $vw-$path is $i$ and the length of the shortest $uw-$path is $i+1$.\\
\begin{center}
    \includegraphics[scale=0.6]{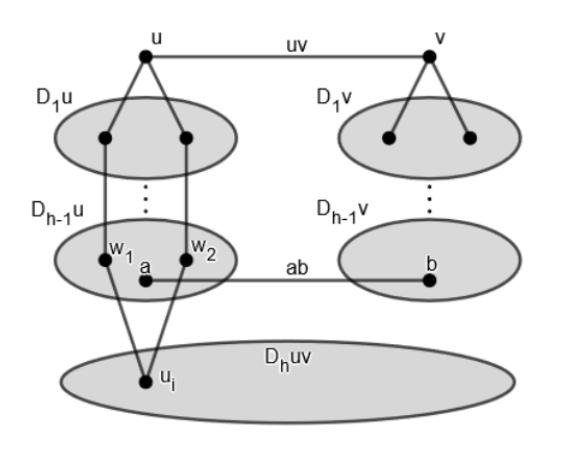}
\end{center}
First, count the number of edges between $D_{h-1}(u)$ and $D_{h-1}(v)$. Suppose that there is an edge $ab$ between $D_{h-1}(u)$ and $D_{h-1}(v)$. Then the edges $ab$, $uv$, and the unique $(h-1)-$paths $av$ and $bv$ form a $g-$cycle. Since $G$ has girth $g$, this is a bijection between the edges between $D_{h-1}(u)$ and $D_{h-1}(v)$ and the $g-$cycles through $uv$. Hence the number of edges between $D_{h-1}(u)$ and $D_{h-1}(v)$ is exactly $\lambda$.

Now, count the number of girth cycles through $u$ that does not contain the edge $uv$.  The graph is $k-$regular, hence $u$ lies on exactly $k$ edges and every edge is contained in exactly $a_i$ distinct $g-$cycles, but in this way, every $g-$cycle through $v$ is counted twice. Therefore the number of $g-$cycles through $u$ is $\frac{1}{2}\sum a_i$. For simplicity, define the variable $s$ as the sum of all $a_i$. So the number of girth cycles through $u$ that does not contain the edge $uv$ is $\frac{s}{2}-\lambda$. There are two types of these cycles: the ones that reach $D_{h-1}(v)$ and the ones that reach the set $D_h(uv):=D_h(u)\cup D_h(v).$\\
We give a lower bound for the number of the first type of these cycles to give an upper bound for the number of the second type of these cycles. There are $(k-1)^{h-1}$ vertices in $D_{h-1}(v)$=$\{v_1,\ldots,v_{(k-1)^{h-1}}\}$. Each vertex $v_i$ has $y_i$ neighbours $D_{h-1}(u)$. Clearly, $\sum y_i=\lambda$. In each vertex $v_i$, we have $\binom{y_i}{2}$ possible choices to form a girth cycle through $u$ that does not contain the edge $uv$. So the number of girth cycles of the first type is
$$\sum_{i=1}^{(k-1)^{h-1}}\binom{y_i}{2}=\frac{1}{2}\sum_{i=1}^{(k-1)^{h-1}}y_i^2-\frac{1}{2}\sum_{i=1}^{(k-1)^{h-1}}y_i=\frac{1}{2}\sum_{i=1}^{(k-1)^{h-1}}y_i^2-\frac{\lambda}{2}.$$

The inequality between the arithmetic and quadratic means gives 
$$\frac{\sum_{i=1}^{(k-1)^{h-1}}y_i^2}{(k-1)^{l-1}}\ge\frac{\left(\sum_{i=1}^{(k-1)^{h-1}}y_i\right)^2}{(k-1)^{2(l-1)}}=\frac{\lambda^2}{(k-1)^{2(h-1)}},$$
hence we obtain the following lower bound:
$$\frac{1}{2}\sum_{i=1}^{(k-1)^{h-1}}y_i^2-\frac{\lambda}{2}\ge\frac{\lambda^2}{2(k-1)^{(h-1)}}-\frac{\lambda}{2}.$$
This lower bound is negative if $\lambda$ is small enough. But we know that the number of girth cycles of the first type is at least zero and an integer, therefore we have the following lower bound:
$$\sum_{i=1}^{(k-1)^{h-1}}\binom{y_i}{2}\ge \max\left(0,\left\lceil\frac{\lambda^2}{2(k-1)^{(h-1)}}-\frac{\lambda}{2}\right\rceil\right).$$
Suppose that there are $m$ vertices that have at least one neighbor in $D_{h-1}(u)$ and are at distance $h$ from the vertex $u$. Their degree set is $\{x_1,\ldots,x_m\}$. We obtain a girth circle if we choose such a vertex $u_i$, its two neighbors, $w_1$ and $w_2$ in $D_{h-1}(u)$ and their unique $(h-1)-$ paths to the vertex $u$. Therefore the number of girth cycles of the second type is exactly $\sum_{i=1}^m\binom{x_i}{2}.$ Now, we have the following upper bound of the number of these circles:
$$\sum_{i=1}^m\binom{x_i}{2}=\frac{s}{2}-\lambda-\sum_{i=1}^{(k-1)^{h-1}}\binom{y_i}{2}\le\frac{s}{2}-\lambda- \max\left(0,\left\lceil\frac{\lambda^2}{2(k-1)^{(h-1)}}-\frac{\lambda}{2}\right\rceil\right).$$
We use this inequality to give a lower bound for $m$ but first, we need to rearrange the terms. We also use the fact that $\sum_{i=1}^mx_i=(k-1)^h-\lambda.$ Now we have that 
$$2\sum_{i=1}^m\binom{x_i}{2}=\sum_{i=1}^mx_i^2-\sum_{i=1}^mx_i\le s-2\lambda- 2\max\left(0,\left\lceil\frac{\lambda^2}{2(k-1)^{(h-1)}}-\frac{\lambda}{2}\right\rceil\right),$$
$$\sum_{i=1}^mx_i^2\le (k-1)^h+s-3\lambda- 2\max\left(0,\left\lceil\frac{\lambda^2}{2(k-1)^{(h-1)}}-\frac{\lambda}{2}\right\rceil\right).$$
By using the inequality between the arithmetic and quadratic means, we get a lower bound for $m$:
$$m\ge\frac{\left(\sum_{i=1}^m x_i\right)^2}{\sum_{i=1}^m x_i^2}\ge \frac{\left((k-1)^h-\lambda\right)^2}{s-3\lambda+(k-1)^h-2\max\left(0,\left\lceil\frac{\lambda^2}{2(k-1)^{(h-1)}}-\frac{\lambda}{2}\right\rceil\right)}.$$
Since $G$ is $k-$regular graph of girth $g$, it has at least $n_0(k,g)$ vertices but with the lower bound of $m$, we also give a lower bound for the additional vertices. We add it to the Moore bound, take the maximum in the signature, and obtain our lower bound for the order of signature-girth-regular graphs of even girth:
$$
\begin{aligned}
        n\ge &2\frac{(k-1)^h-1}{k-2}\\
        &+\max_{1\le i\le k}\left\lceil\frac{\left((k-1)^h-a_i\right)^2}{\sum_{j=1}^ka_j-3a_i+(k-1)^h-2\max\left(0,\left\lceil\frac{a_i^2}{2(k-1)^{(h-1)}}-\frac{a_i}{2}\right\rceil\right)}\right\rceil.
    \end{aligned}
$$
\end{proof}
\begin{corollary}\label{new_bound_egr}
If $G$ is an $egr(n,k,g,\lambda)$ graph with even girth $g=2h$, then
$$n\ge 2\frac{(k-1)^{h}-1}{k-2}+\left\lceil\frac{((k-1)^{h}-\lambda)^2}{(k-3)\lambda+(k-1)^{h}-2\max\left(0,\left\lceil\frac{\lambda^2}{2(k-1)^{h-1}}-\frac{\lambda}{2}\right\rceil\right)}\right\rceil.$$ 
\end{corollary}
\begin{remark}
{\rm The term 
$\max\left(0,\left\lceil\frac{\lambda^2}{2(k-1)^{h-1}}-\frac{\lambda}{2}\right\rceil\right)$
equals $0$ for every $\lambda\le (k-1)^{h-1}.$}
\end{remark}

Finally, we compare the existing lower bounds (DFJR21 \cite{bound}, P23 \cite{po}) on the order of extremal edge-girth-regular graphs of even girth with the one in Corollary \ref{new_bound_egr}.\\
\begin{center}
\includegraphics[scale=0.8]{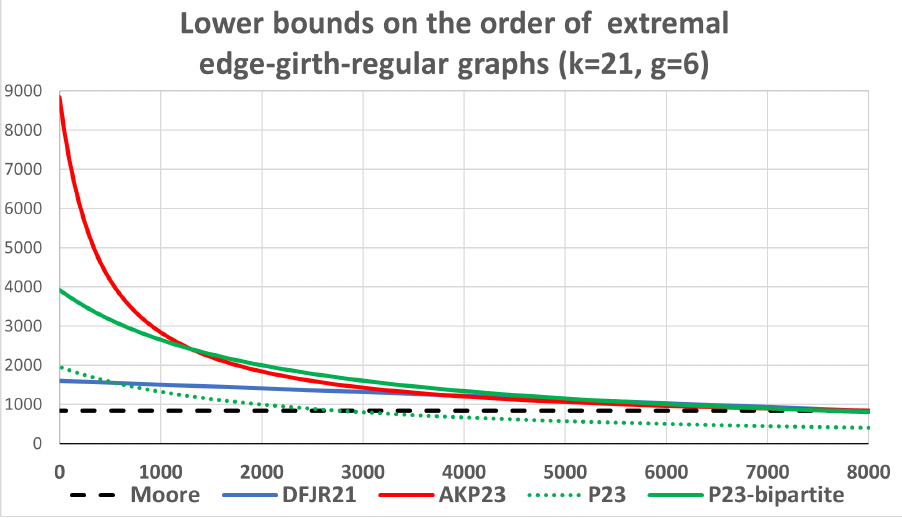}
\end{center}
When $\lambda$ is close to its upper bound $(k-1)^ \frac{g}{2}$, then all three lower bounds are close to the Moore-bound, DFJR21 is always the best one. But it falls back when $\lambda$ is small, as we have seen in Remark \ref{compare_4}. Our result and P23 have the same order of magnitude, namely $\Theta(k^{\frac g2})$, when $\lambda$ is in the order of magnitude $O(k^{\frac g2-1}),$ however our result is better with a constant multiplier. To compare it to DFJR21, it has order of magnitude $\Theta(k^{\frac g2-1})$ for small $\lambda$ parameters.

\noindent
{\bf Gabriela Araujo-Pardo:} 
Instituto de Matem\'aticas-Campus Juriquilla, Universidad Nacional Aut\'onoma de 
M\'exico, C.P. 076230, Boulevard Juriquilla \# 3001, Juriquilla, Qro., M\'exico;  \\
 e-mail:  \texttt{garaujo@im.unam.mx} \\

\noindent
{\bf Gy\"orgy Kiss:} Department of Geometry and HUN-REN-ELTE Geometric and Algebraic Combinatorics
Research Group, E\"otv\"os Lor\'and University, 1117 Budapest, P\'azm\'any
s. 1/c, Hungary; and Faculty of Mathematics, Natural Sciences and Information Technologies, University of Primorska, Glagolja\v ska 8, 6000 Koper,
Slovenia; \\
e-mail: \texttt{gyorgy.kiss@ttk.elte.hu} \\

\noindent
{\bf Istv\'an Porups\'anszki:} Institute of Mathematics and HUN-REN-ELTE Geometric and Algebraic Combinatorics
Research Group, E\"otv\"os Lor\'and University, 1117 Budapest, P\'azm\'any
s. 1/c, Hungary; \\ 
e-mail: \texttt{rupsansz@gmail.com}


\begin{thebibliography}{99} 

\bibitem{AABL11} M. Abreu,  G. Araujo-Pardo, C. Balbuena, D. Labbate. Families of Small Regular Graphs of Girth $5$.
{\em Discrete Math.} \textbf{312} (2012), 2832--2842  

\bibitem{AABB17} E. Abajo G. Araujo, C. Balbuena, M. Bendala. New small regular graphs of girth $5$. {\em Discrete Math.} \textbf{340} (2017), 1878--1888.

\bibitem{ABBM19} E. Abajo, C. Balbuena, M. Bendala. X. Marcote. Improving bounds on the order of regular graphs of girth $5$.  {\em Discrete Math.} \textbf{342} (2019), 2900--2910.

\bibitem{AB21} E. Abajo, M. Bendala. Regular graphs of girth $5$ from elliptic semi planes of type C.  {\em Discrete Math.} \textbf{344} (2021), Paper 112343.   

\bibitem{AL22}
G. Araujo-Pardo and D. Leemans, Edge-girth-regular graphs arising from
biaffine planes and Suzuki groups, \emph{Discrete Math.} \textbf{345} (2022), 
Paper 112991.

\bibitem{Br-5}
 W. G. Brown, On the non-existence of a type of regular graphs of girth 5, 
\emph{Canad. J. Math.} \textbf{19} (1967), 644--648.

\bibitem{bound} A. Z. Drglin, S. Filipovski, R. Jajcay, and T. Raiman, Extremal Edge-Girth-Regular Graphs, \emph{Graphs and Combin.} \textbf{37} (2021), 2139–2154.
 
\bibitem{ex-126}
G. Exoo, Regular graphs of given degree and girth.
http://ginger.indstate.edu/ge/CAGES


\bibitem{ExooJaj08}
G.\ Exoo and R.\ Jajcay,
Dynamic cage survey,
{\em Electron. J. Combin., Dynamic Survey} {\bf 16} (2008).
    
\bibitem{GH08}
A. Gács, T. Héger, On geometric constructions of (k,g)-graphs, Contrib. Discrete Math. 
\textbf{3}  (2008), 63--80.

\bibitem{jkm}
R. Jajcay, Gy. Kiss and  \v S. Miklavi\v c, Edge-girth-regular graphs, \emph{European J. Combin.} {\bf 72} (2018), 70--82.

\bibitem{jorg}
L. K. J\o rgensen, Girth 5 graphs from relative difference sets, \emph{Discrete Math.} 
{\bf 293} (2005), 177--184.


\bibitem{KSz}
Gy.~Kiss, T.~Sz\H onyi, \emph{Finite Geometries,} CRC Press,Taylor \& Francis Group, 2019.

\bibitem{PT}
S.~E.~Payne, J.~A.~Thas, \emph{Finite Generalized Quadrangles}, Second Ed., 
EMS Publishing House, 2009.

\bibitem{po}
I. Porupsánszki, \emph{On edge-girth-regular graphs: lower bounds and new families}, submitted, 2022.

\bibitem{pv} 
P. Poto\v cnik and J. Vidali, Girth-regular graphs, \emph{Ars Math. Contemp.} {\bf 17} (2019), 349--368.

\bibitem{wo}
P.~K. Wong, \emph{On the uniqueness of the smallest graphs of girth 5 and valency 6,}
J. Graph Theory {\bf 3} (1978), 407--409. 

\end{thebibliography}
\end{document}